\documentclass[11pt]{article}
\usepackage[margin=1.02in]{geometry}
\usepackage{amsmath,amssymb,amsthm,mathtools}
\usepackage{microtype}
\usepackage{enumitem}
\usepackage[colorlinks=true,linkcolor=black,citecolor=blue,urlcolor=blue]{hyperref}
\usepackage{tikz}
\usetikzlibrary{arrows.meta,calc,decorations.pathmorphing}
\definecolor{figblue}{RGB}{35,105,180}
\definecolor{figred}{RGB}{200,55,55}
\definecolor{figgreen}{RGB}{35,145,90}
\usetikzlibrary{arrows.meta,positioning}
\setlist{nosep}
\newtheorem{theorem}{Theorem}[section]
\newtheorem{proposition}[theorem]{Proposition}
\newtheorem{corollary}[theorem]{Corollary}

\theoremstyle{remark}
\newtheorem{remark}[theorem]{Remark}
\newtheorem{example}[theorem]{Example}
\newtheorem{definition}[theorem]{Definition}
\newcommand{\Z}{\mathbb Z}

\newcommand{\SympSum}{\mathbin{\#}}
\title{Symplectic Surface Summing and Negative Spheres}
\author{Anar Akhmedov}
\date{}

\begin{document}
\maketitle

\begin{abstract}
\noindent We extend the sphere-summing construction of \cite{AZ} from torus sums of spheres to symplectic sums along surfaces of arbitrary genus. The relative surfaces may have arbitrary positive intersection number with the summing surface. We give formulas for the genus and self-intersection of the resulting surface, together with graph and simultaneous versions of the construction. The original elliptic-surface case is recovered as a special case, and higher-genus examples are obtained from hyperelliptic Lefschetz fibrations and braided symplectic surfaces in $\Sigma_h\times S^2$.
\end{abstract}

\section{Introduction}

Our starting point is the \emph{summing spheres} construction from our joint
work with Weiyi Zhang \cite{AZ}, which dates back to 2012.  In \cite{AZ}, several symplectic spheres
$S_1,\ldots,S_p$, each meeting a square-zero symplectic torus $T$ once, are
joined using a fiber sum with a connected $p$-fold braided torus in
$\mathbb{T}^2\times \mathbb{S}^2$. It was shown that the result is a symplectic sphere of square
\[
        S_1^2+\cdots+S_p^2.
\]
For the particular braid and gluing used in \cite{AZ}, the construction also produces symplectic $4$-manifolds with cyclic fundamental group and negative spheres suitable for rational blowdown. Combined with Luttinger surgery and symplectic sums, this gave examples with $b^+=1$ and small $c_1^2$. The article is available at \href{https://arxiv.org/abs/1506.08367}{\texttt{arXiv:1506.08367}}.

Neither restriction is essential. The relative surfaces may have positive genus and may meet the summing surface more than once; the summing surface itself may also have arbitrary genus.  Let $\Sigma\subset X$ and $\Sigma'\subset Z$ be
symplectic surfaces of the same genus, with opposite self-intersection, and
assume that their symplectic forms have been adjusted so that the symplectic
sum is defined.  Suppose that $C_1,\ldots,C_r\subset X$ are embedded
symplectic surfaces, pairwise disjoint away from $\Sigma$, with
\[
        C_i\cdot\Sigma=d_i>0,\qquad m=\sum_{i=1}^{r}d_i,
\]
and that an embedded symplectic surface $F\subset Z$ satisfies
$F\cdot\Sigma'=m$.  With a surface-compatible choice of the symplectic-sum
gluing, the punctured surfaces match to a connected embedded symplectic
surface $C\subset X\#_{\Sigma=\Sigma'}Z$.  The self-intersection and genus are
\[
 C^2=F^2+\sum_{i=1}^r C_i^2,\qquad
 g(C)=g(F)+\sum_{i=1}^r g(C_i)+m-r.
\]
There is also a useful graph formulation: vertices represent the relative surfaces and edges the matched boundary circles.  For a connected gluing graph $\Gamma$,
\[
        g(C_\Gamma)=\sum_{v\in V(\Gamma)}g(C_v)+b_1(\Gamma).
\]
Thus each independent cycle contributes one to the genus. The same construction may be performed simultaneously on disjoint collections of relative surfaces; this will be useful for the rational blowdown configurations below.

Related constructions appear in our work with Burak Ozbagci \cite{AO} and N.~Saglam \cite{ASaglam}. In \cite{AO}, knot surgery was incorporated into Lefschetz fibrations with arbitrary finitely presentable fundamental group. In \cite{ASaglam}, Luttinger surgery on Lefschetz fibrations was used to construct small exotic symplectic $4$-manifolds. The exceptional sections arising in these constructions are also relevant here; see the recent work with S.~Sakall{\i} \cite{ASGurtas}.

The idea of the symplectic normal connected sum goes back to Gromov's
work \cite[Section~3.4.4]{Gromov}; see also McCarthy--Wolfson \cite{MW2},
where this attribution is made.  The construction was subsequently
developed and used by Gompf \cite{Gompf} and by McCarthy--Wolfson
\cite{MW}.
Fintushel and Stern later used closed braids in a tubular neighborhood of a
torus to construct connected symplectic representatives of multiples of a
torus class \cite{FSsurf}.  Smith later
developed related constructions from surface fibrations, including
higher-genus symplectic representatives \cite{SmithSurface} in ruled surfaces.

\par\vspace{0.8\baselineskip}
\noindent\textbf{Acknowledgment.}\enspace The sphere-summing construction first appeared in joint work with Weiyi
Zhang in 2012. I am grateful to Weiyi Zhang, Burak Ozbagci, and S\"umeyra
Sakall{\i} for the collaborations discussed above \cite{AZ,AO,ASGurtas}.
Several of the ideas in this note go back to earlier work and are taken up here again after an interruption. Discussions arising from my collaboration with Sakall{\i} on
G\"urta\c{s} Lefschetz fibrations helped bring me back to these questions. An LLM-based tool was used to assist with the preparation, including grammar and language editing, and with the generation of figures. All mathematical content is due to the author, who takes full intellectual responsibility for the content of this paper.
\section{The relative surface-summing construction}
Let $\Sigma\subset X$ and $\Sigma'\subset Z$ be closed symplectic surfaces
of the same genus, with opposite self-intersection, and assume that their
symplectic forms have been adjusted so that the symplectic sum is defined.
Thus
\[
                    \Sigma^2=-(\Sigma')^2.
\]
The symplectic normal connect sum \cite{MW,Gompf} identifies the oriented normal circle bundles by a
fiber-orientation-reversing bundle isomorphism and produces
\[
                    Y=X\SympSum_{\Sigma=\Sigma'}Z.
\]
The square-zero torus case used in \cite{AZ} is recovered by taking
$\Sigma=\Sigma'=T^2$.
Let $C_1,\ldots,C_r\subset X$ be embedded symplectic surfaces such that:
\begin{enumerate}
\item $C_i$ and $C_j$ are disjoint for $i\ne j$ away from $\Sigma$;
\item each $C_i$ meets $\Sigma$ transversely and positively in $d_i\geq1$
points;
\item all points in the sets $C_i\cap\Sigma$ are distinct.
\end{enumerate}
Let $F\subset Z$ be an embedded symplectic surface meeting $\Sigma'$
transversely and positively in
\[
                    m=\sum_{i=1}^r d_i
\]
distinct points.
Deleting the tubular neighborhoods of $\Sigma$ and $\Sigma'$ removes a small
disk from the corresponding relative surface at each intersection point.
Near such a point the two symplectic surfaces may be put in the standard
positive transverse local model.  The boundary circle of each punctured
surface then carries the normal framing coming from this local product model.
\begin{definition}
A symplectic-sum gluing is called \emph{surface-compatible} for the chosen
matching if it sends the boundary circle associated to each point
$C_i\cap\Sigma$ to the prescribed boundary circle associated to
$F\cap\Sigma'$, with the induced orientations and normal framings matched.
\end{definition}
We shall always assume that the gluing is surface-compatible. To arrange this, choose an orientation-preserving diffeomorphism $\Sigma\to\Sigma'$ carrying the marked points to their prescribed partners. Since the normal bundles have opposite Euler classes, it lifts to a fiber-orientation-reversing isomorphism of their circle bundles. Local fiber rotations near the marked points can then be used to match the product framings.
The resulting closed
surface is denoted
\[
                    C=C_1\#\cdots\# C_r\# F.
\]
\begin{theorem}[Surface-summing theorem]\label{thm:main}
With the hypotheses above and a surface-compatible gluing, the symplectic
sum $Y=X\SympSum_{\Sigma=\Sigma'}Z$ contains an embedded connected symplectic
surface $C$
such that
\begin{align}
 g(C)&=g(F)+\sum_{i=1}^r g(C_i)+m-r,\label{eq:genus}\\
 C^2&=F^2+\sum_{i=1}^r C_i^2.\label{eq:square}
\end{align}
The construction is local near the surfaces and is compatible with any further
symplectic surfaces or configurations disjoint from the summing neighborhoods.
\end{theorem}
\begin{proof}
Choose symplectic tubular neighborhoods at the intersection points. After a local isotopy, $(\Sigma,C_i)$ and $(\Sigma',F)$ may be taken to be the coordinate axes in the standard positive transverse model. The surface-compatible gluing identifies the corresponding boundary circles and framings, and the usual smoothing joins the punctured pieces to an embedded symplectic surface. It is connected because every $C_i$ is joined to $F$.
For the genus, write $g_i=g(C_i)$ and $h=g(F)$.  Removing $d_i$ open disks
from $C_i$ gives Euler characteristic $2-2g_i-d_i$, while removing $m$ disks
from $F$ gives $2-2h-m$.  Gluing along circles does not change Euler
characteristic.  Hence
\[
 \chi(C)=2-2h-m+\sum_{i=1}^r(2-2g_i-d_i)
        =2+2r-2h-2\sum_i g_i-2m.
\]
Since $C$ is connected, $\chi(C)=2-2g(C)$, which gives
\eqref{eq:genus}.
For the square, use the relative normal framings coming from the local product models. These framings extend over the deleted disks, so the relative normal Euler numbers of the punctured $C_i$ and $F$ are $C_i^2$ and $F^2$. Normal push-offs can be chosen constant near the boundary and hence glue across the necks. The necks contribute zero, which gives \eqref{eq:square}.
\end{proof}
\begin{remark}

The theorem only determines the resulting surface; it does not determine the fundamental group of the ambient fiber sum.  The fundamental group
of $Y$ depends on the complements of $\Sigma$ and $\Sigma'$ and on the
gluing map.  In the torus case it can also be sensitive to the chosen braided
torus.
\end{remark}
\section{The graph formula and simultaneous summing} 
Theorem~\ref{thm:main} has the following graph formulation.
Let $\Gamma$ have vertices corresponding to $F,C_1,\ldots,C_r$, and an edge
for each pair of boundary circles glued together.  In the situation above,
$\Gamma$ has $r+1$ vertices and $m$ edges.
\begin{proposition}[Graph genus formula]\label{prop:graph}
Consider an iterated sequence of symplectic sums in which embedded symplectic
surfaces are joined only by surface-compatible gluings as above.  If the associated
gluing graph $\Gamma$ is connected, the resulting connected surface
$C_\Gamma$ satisfies
\[
 g(C_\Gamma)=\sum_{v\in V(\Gamma)}g(C_v)+b_1(\Gamma).
\]
If the normal framings are compatible at every gluing, then
\[
 C_\Gamma^2=\sum_{v\in V(\Gamma)} C_v^2.
\]
\end{proposition}
\begin{proof}
If $\Gamma$ has $V$ vertices and $E$ edges, puncturing the vertex surfaces once
for each incident edge decreases the sum of their Euler characteristics by
$2E$.  Gluing the resulting boundary circles changes no Euler characteristic.
Thus
\[
 2-2g(C_\Gamma)=\sum_v(2-2g(C_v))-2E.
\]
Since $\Gamma$ is connected, $b_1(\Gamma)=E-V+1$, and the genus formula follows.
The square formula is the same normal-Euler-number argument used in
Theorem~\ref{thm:main}.
\end{proof}
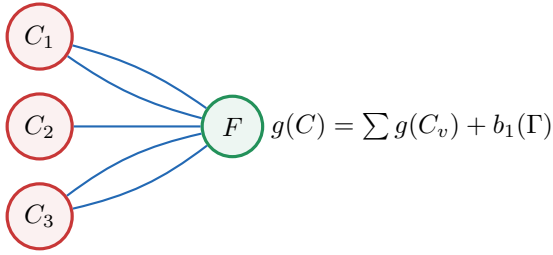
\begin{center}
\begin{tikzpicture}[scale=.85, every node/.style={font=\small}]
\node[circle,draw=figgreen,very thick,fill=figgreen!8,minimum size=8mm] (F) at (0,0) {$F$};
\node[circle,draw=figred,very thick,fill=figred!7,minimum size=8mm] (C1) at (-3,1.4) {$C_1$};
\node[circle,draw=figred,very thick,fill=figred!7,minimum size=8mm] (C2) at (-3,0) {$C_2$};
\node[circle,draw=figred,very thick,fill=figred!7,minimum size=8mm] (C3) at (-3,-1.4) {$C_3$};
\draw[figblue,thick] (C1) to[bend left=10] (F);
\draw[figblue,thick] (C1) to[bend right=10] (F);
\draw[figblue,thick] (C2) -- (F);
\draw[figblue,thick] (C3) to[bend left=12] (F);
\draw[figblue,thick] (C3) to[bend right=12] (F);
\node at (2.8,0) {$g(C)=\sum g(C_v)+b_1(\Gamma)$};
\end{tikzpicture}
\smallskip
{\small The gluing graph records multiple intersections; parallel edges
create genus in the summed surface.}
\end{center}
The construction can also be carried out simultaneously for several disjoint surfaces.  Suppose the surfaces
$C_1,\ldots,C_r$ are partitioned into nonempty sets
$P_1,\ldots,P_s$.  Assume that the second summand contains symplectic connector surfaces
$F_1,\ldots,F_s$ which are pairwise disjoint away from $\Sigma'$ and whose
intersection points with $\Sigma'$ are mutually distinct, where
\[
 F_j\cdot \Sigma'=\sum_{i\in P_j}d_i.
\]
Choose the matching so that the punctures of $F_j$ are glued exactly to those
of the $C_i$ with $i\in P_j$.
\begin{corollary}[Simultaneous surface summing]\label{cor:simult}
Under the preceding hypotheses, and for surface-compatible matchings chosen
separately for the blocks $P_j$, the fiber sum contains pairwise disjoint
symplectic surfaces $D_1,\ldots,D_s$ with
\begin{align*}
 D_j^2&=F_j^2+\sum_{i\in P_j}C_i^2,\\
 g(D_j)&=g(F_j)+\sum_{i\in P_j}g(C_i)
       +\sum_{i\in P_j}d_i-|P_j|.
\end{align*}
\end{corollary}
If the resulting squares are negative, we obtain several disjoint negative surfaces in the same fiber sum. Disjointness is built into the choice of the connector surfaces and punctures.
\section{Recovery of the sphere-summing construction}
We recover the construction of \cite{AZ} by taking the summing surface to be
a torus.  In the original construction, the first summand is an elliptic
surface $E(n)$ and the summing torus is a regular elliptic fiber
$T\subset E(n)$.  The second summand is $\mathbb{T}^2\times \mathbb{S}^2$, containing a connected
symplectic $p$-fold braided torus $T_p$ representing
$p[\mathbb{T}^2\times\{\mathrm{pt}\}]$.  If
\[
 S_1,\ldots,S_p
\]
are disjoint sphere sections of $E(n)$, then each $S_i$ meets $T$ transversely
once.  On the second side the sphere
\[
 F=\{\mathrm{pt}\}\times S^2
\]
meets $T_p$ in $p$ points.  After removing neighborhoods of $T$ and $T_p$,
the $p$ punctured sections are joined to the $p$ punctures of $F$ across the
fiber-sum neck.  Applying Theorem~\ref{thm:main} with
$r=m=p$, $d_i=1$, and $g(S_i)=g(F)=0$ gives
\[
 g(C)=0,\qquad C^2=\sum_{i=1}^p S_i^2.
\]
Thus the output is again a symplectic sphere.  In particular, since a section
of $E(n)$ has square $-n$, $p$ such sections give a sphere of square $-pn$.
There is a direct higher-genus analogue of the first summand.  Let
\[
 f_g\colon X(g,1)=\mathbb{CP}^2\#(4g+5)\overline{\mathbb{CP}}^{\,2}
 \longrightarrow S^2
\]
be the standard hyperelliptic genus-$g$ Lefschetz fibration.  It admits
$4g+4$ disjoint $(-1)$-sphere sections \cite{Tanaka}.  Forming the fiber sum
of $n$ copies along regular genus-$g$ fibers,
\[
 X(g,n)=\underbrace{X(g,1)\#_{\Sigma_g}\cdots
 \#_{\Sigma_g}X(g,1)}_{n\ \mathrm{copies}},
\]
and gluing corresponding sections gives $4g+4$ disjoint sphere sections of
square $-n$.  Hence $X(g,n)$ plays the role of $E(n)$ in higher genus. If the other summand contains a genus-$g$ summing surface with a connector meeting it in $p$ positive points, the same construction joins $p$ sections to a symplectic sphere of square $-pn$.
 In recent work with S\"umeyra Sakall{\i}, we
study exceptional sections, braided bisections, negative spheres in fiber
sums of G\"urta\c{s} Lefschetz fibrations, and their use in constructions of
exotic symplectic $4$-manifolds \cite{ASGurtas}.
\begin{corollary}[Akhmedov--Zhang sphere summing (2012)]\label{cor:AZ}
The sphere-summing surface in \cite{AZ} is the $d_i=1$, genus-zero case of
Theorem~\ref{thm:main}.  For the particular braided torus and gluing used in
\cite{AZ}, the resulting manifold has fundamental group $\Z_p$ under the
complement hypotheses imposed there.
\end{corollary}
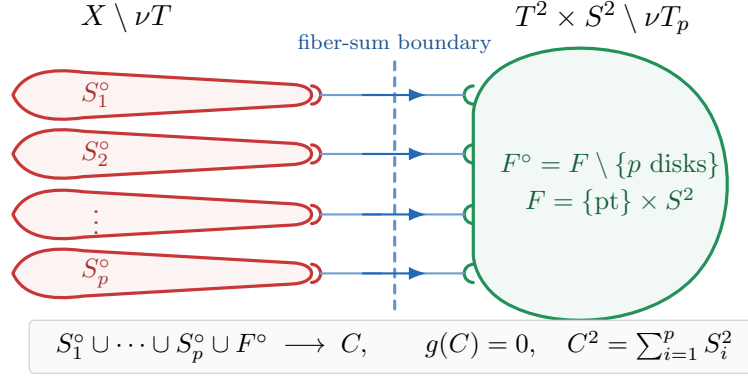
\begin{figure}[ht]
\centering
\begin{tikzpicture}[x=1cm,y=1cm,>=Latex,font=\small,
                    line cap=round,line join=round]
\node[font=\bfseries] at (-3.55,2.20) {$X\setminus\nu T$};
\node[font=\bfseries] at (2.75,2.20) {$T^2\times S^2\setminus\nu T_p$};

\draw[figblue!70,very thick,dashed] (0,-1.65)--(0,1.65);
\node[figblue!80!black,font=\scriptsize,fill=white,inner sep=1.5pt]
      at (0,1.86) {fiber-sum boundary};

\foreach \y/\lab in {1.18/$S_1^\circ$,.40/$S_2^\circ$,-.40/$\vdots$,-1.18/$S_p^\circ$}{
  \draw[figred,very thick,fill=figred!6]
    (-5.05,\y) .. controls (-4.92,\y+.30) and (-4.55,\y+.36) .. (-4.12,\y+.30)
    -- (-1.35,\y+.14)
    .. controls (-1.20,\y+.13) and (-1.10,\y+.08) .. (-1.10,\y)
    .. controls (-1.10,\y-.08) and (-1.20,\y-.13) .. (-1.35,\y-.14)
    -- (-4.12,\y-.30)
    .. controls (-4.55,\y-.36) and (-4.92,\y-.30) .. (-5.05,\y) -- cycle;
  \node[figred!85!black] at (-3.95,\y) {\lab};
  \draw[figred,very thick] (-1.10,\y-.12) arc[start angle=-90,end angle=90,radius=.12];
  \draw[figblue!65,thick] (-.98,\y)--(-.18,\y);
}

\draw[figgreen,very thick,fill=figgreen!6]
  (1.45,1.56)
  .. controls (2.10,1.95) and (3.34,1.88) .. (3.95,1.24)
  .. controls (4.55,.62) and (4.55,-.62) .. (3.95,-1.24)
  .. controls (3.34,-1.88) and (2.10,-1.95) .. (1.45,-1.56)
  .. controls (1.18,-1.40) and (1.04,-1.12) .. (1.04,-.92)
  -- (1.04,.92)
  .. controls (1.04,1.12) and (1.18,1.40) .. (1.45,1.56) -- cycle;

\foreach \y in {1.18,.40,-.40,-1.18}{
  \draw[figgreen,very thick] (1.04,\y-.12) arc[start angle=270,end angle=90,radius=.12];
  \draw[figblue!65,thick] (.18,\y)--(.92,\y);
  \draw[->,figblue,thick] (-.42,\y)--(.42,\y);
}
\node[figgreen!65!black,align=center] at (2.85,0)
      {$F^\circ=F\setminus\{p\text{ disks}\}$\\[2pt]
       $F=\{\mathrm{pt}\}\times S^2$};

\node[draw=black!20,rounded corners=2pt,fill=black!2,
      inner xsep=10pt,inner ysep=4pt] at (0,-2.15)
      {$S_1^\circ\cup\cdots\cup S_p^\circ\cup F^\circ
        \;\longrightarrow\; C,\qquad
        g(C)=0,\quad C^2=\sum_{i=1}^p S_i^2$};
\end{tikzpicture}
\caption{Sphere summing.  The punctured sections $S_i^\circ$ are matched to the boundary components of the punctured connector $F^\circ$ across the fiber-sum neck.  The resulting closed surface is the symplectic sphere $C$, with $g(C)=0$ and $C^2=\sum_{i=1}^p S_i^2$.}
\label{fig:summing-spheres}
\end{figure}
The details of fundamental-group computation can be found in \cite{AZ}. Note that the cyclic fundamental group is obtained from an explicit van Kampen
calculation for the chosen braided torus.  A different braid or gluing map can change the ambient fundamental group without changing the genus or square of
the summed surface. When the $S_i$ are exceptional spheres, the construction has the following immediate consequence.
\begin{example}\label{ex:minus-p}
If $S_1,\ldots,S_p$ are disjoint $(-1)$-spheres meeting $T$ once, sphere
summing produces a symplectic $(-p)$-sphere.  This recovers the basic
mechanism used in \cite{AZ} to obtain $(-4)$, $(-5)$, and $(-6)$-spheres
before rational blowdown.  For example, four $(-1)$-sections were summed to
a $(-4)$-sphere; rationally blowing it down gave a symplectic $4$-manifold
with $\pi_1=\mathbb Z_4$ and $c_1^2=1$.  Starting instead with eight
$(-1)$-sections, divided into two groups of four, gave two disjoint
$(-4)$-spheres, and rationally blowing down both produced an example with
$\pi_1=\mathbb Z_4$ and $c_1^2=2$.  The same construction also gives examples
with $\pi_1=\mathbb Z_6$ and $c_1^2=3$.
\end{example}
More generally, if the $C_i$ are spheres of squares $-a_i$ and meet $T$ once,
then their sum is a sphere of square $-\sum a_i$.  Thus a collection of mild
negative classes can be consolidated into a single class of larger negative
square without changing $e$ or $\sigma$ of the ambient manifold when the
second summand is $T^2\times S^2$.

\section{Negative configurations and rational blowdown}

We now use the sphere-summing construction to produce the negative
configurations needed for rational blowdown \cite{FSRB}.  Suppose that $X$ contains
disjoint exceptional spheres
\[
E_1,\ldots,E_N,
\]
each meeting the square-zero torus $T$ once.  Partition these spheres into
blocks of sizes $p_1,\ldots,p_s$, where $\sum_{j=1}^s p_j=N$.  If the second summand
contains disjoint connector spheres meeting the summing torus in $p_j$ points,
Corollary~\ref{cor:simult} gives disjoint symplectic spheres
\[
D_j^2=-p_j.
\]

Intersections away from the summing torus are unchanged. Thus a sphere disjoint from $T$ survives, with its intersections with the $E_i$ transferred to the corresponding $D_j$. In this way the original incidence data in $X$ gives linear or star-shaped negative configurations.

\begin{proposition} Let $\mathcal C$ be a symplectic sphere configuration in $X$, and suppose
that its intersections with $E_1,\ldots,E_N$ occur away from $T$.  After a
surface-compatible simultaneous sphere sum of the $E_i$ in blocks, the
components of $\mathcal C$ disjoint from the summing neighborhood survive.
A block of $p$ exceptional spheres is replaced by a sphere of square $-p$,
and its algebraic intersection number with each surviving component is the sum of the
corresponding intersection numbers of the spheres in that block.
\end{proposition}

\begin{proof}
The self-intersection calculation follows from
Corollary~\ref{cor:simult}.  Since all relevant intersections occur outside
the tubular neighborhood removed in forming the fiber sum, they are
unchanged by the gluing and become intersections with the summed sphere.
\end{proof}

For explicit examples of this construction and the resulting rational
blowdowns, we refer to our earlier work \cite{AZ}. Further examples arising from higher-genus Lefschetz fibrations and Smith's braided surfaces are developed in \cite{ASGurtas}.

\section{Braided tori and higher-genus braided surfaces}
We recall the braided torus used in the original construction.
Write a tubular neighborhood of the product torus as
\[
             \nu T=T^2\times D^2=S^1_x\times(S^1_y\times D^2).
\]
Let \(B\subset S^1_y\times D^2\) be a connected closed \(m\)-strand
braid.  Equivalently, the projection
\[
             B\longrightarrow S^1_y
\]
has degree \(m\).  Taking the product with the remaining circle gives
\[
             T_B=S^1_x\times B\subset T^2\times D^2.
\]
Thus the projection \(T_B\to T^2=S^1_x\times S^1_y\) has degree \(m\), and
\[
             [T_B]=m[T]\in H_2(T^2\times D^2).
\]
For a braid positively transverse to the disk fibers, the product torus
\(T_B\) is symplectic after choosing the standard product symplectic form.
This is the braided-torus construction used by Fintushel and Stern
\cite{FSsurf}.  Braided surfaces also appeared in earlier joint work with B.~Doug Park
\cite{APodd}.  There a connected two-string braid
$\beta\subset D^2\times S^1$ gives the symplectic torus
\[
T_\beta=\beta\times S^1\subset T^4,\qquad
[T_\beta]=2[\alpha_3\times\alpha_4].
\]
The torus $T_\beta$ meets $\alpha_1\times\alpha_2$ in two points; resolving
one intersection and blowing up the remaining double point produces the
genus-two symplectic surface used in the constructions of exotic
$\mathbb{CP}^2\#2\overline{\mathbb{CP}}^{\,2}$,
$\mathbb{CP}^2\#4\overline{\mathbb{CP}}^{\,2}$, and other small exotic
$4$-manifolds.
In this construction $T_B$ is the \emph{summing surface}. After placing it in $T^2\times S^2$, the transverse sphere
\[
             F=\{x_0\}\times S^2
\]
meets \(T_B\) positively in exactly \(m\) points.  Removing a neighborhood of
\(T_B\) therefore turns \(F\) into a sphere with \(m\) boundary components.
Its boundary circles are matched with those of the punctured surfaces on the $X$-side. In \cite{AZ}, these are $m$ spheres meeting the summing torus once, so the punctured sphere $F$ joins them to a single sphere. If the connector has genus $h$ instead, Theorem~\ref{thm:main} gives
\[
             g(C)=h+\sum_{i=1}^r g(C_i)+m-r,
\]
with the same additivity for the normal Euler number. The case $h=0$ and $r=m$ is the original sphere-summing situation.

For higher genus we use Smith's surfaces. Proposition~1.2 of
\cite{SmithSurface} says that, for every
\(h\geq 1\) and \(k\geq 2\), the class
\[
             2k[\Sigma_h]\in H_2(\Sigma_h\times S^2;\mathbb Z)
\]
has infinitely many pairwise non-isotopic connected symplectic
representatives
\[
             B_{h,k,a}\subset \Sigma_h\times S^2 ,
             \qquad a\in\mathbb Z,
\]
each of genus
\[
             g(B_{h,k,a})=1+2k(h-1).
\]
These surfaces have square zero.  Moreover the sphere fiber
\[
             F=\{x\}\times S^2
\]
meets \(B_{h,k,a}\) positively in \(2k\) points, since
\([B_{h,k,a}]=2k[\Sigma_h]\).
Theorem~\ref{thm:main} gives the following.
\begin{proposition}
\label{prop:smith-summing}
Fix \(h\geq1\) and \(k\geq2\), and let
\[
                     r=1+2k(h-1).
\]
Let \(X\) contain a square-zero symplectic surface
\(\Sigma\) of genus \(r\), together with \(2k\) symplectic spheres
\(S_1,\ldots,S_{2k}\), pairwise disjoint away from \(\Sigma\), each meeting
\(\Sigma\) transversely and positively once.  For any Smith representative
\(B_{h,k,a}\subset\Sigma_h\times S^2\), let us form a surface-compatible symplectic
sum
\[
 Y_a=X\#_{\Sigma=B_{h,k,a}}(\Sigma_h\times S^2).
\]
Then \(Y_a\) contains a symplectic sphere \(C_a\) with self-intersection
\[
                     C_a^2=\sum_{i=1}^{2k} S_i^2.
\]
In particular, if all \(S_i\) are exceptional spheres, then
\[
                     C_a^2=-2k.
\]
\end{proposition}
\begin{proof}
The sphere fiber $F=\{x\}\times S^2$ meets $B_{h,k,a}$ in $2k$ positive points. After deleting the summing neighborhoods, $F$ has $2k$ boundary circles and each $S_i$ has one. Matching these circles and applying Theorem~\ref{thm:main} gives
\[
 g(C_a)=0+\sum_{i=1}^{2k}0+2k-2k=0
\]
and
\[
 C_a^2=F^2+\sum_{i=1}^{2k}S_i^2
       =\sum_{i=1}^{2k}S_i^2,
\]
since \(F^2=0\).
\end{proof}
Thus the genus of the summing surface may be arbitrarily large while the connector remains a sphere. Smith's construction supplies infinitely many non-isotopic summing surfaces in the same homology class, and the above calculation is unchanged. We make no claim that the corresponding $Y_a$, or the spheres $C_a$, are pairwise distinct; that question depends on the complements and the gluing maps.
\par\vspace{0.35\baselineskip}

\par\vspace{0.8\baselineskip}
\noindent\textbf{Acknowledgment.}\enspace The sphere-summing construction first appeared in joint work with Weiyi
Zhang in 2012. I am grateful to Weiyi Zhang, Burak Ozbagci, and S\"umeyra
Sakall{\i} for the collaborations discussed above \cite{AZ,AO,ASGurtas}.
Several of the ideas in this note go back to earlier work and are taken up here again after an interruption. Discussions arising from my collaboration with Sakall{\i} on G\"urta\c{s} Lefschetz fibrations helped bring me back to these questions. An LLM-based tool was used to assist with the preparation of portions of the text, including grammar and language editing, and with the generation of the figures. All mathematical content is due to the author, who takes full intellectual responsibility for the content of this paper.

\end{document}